\documentclass[12pt,letterpaper,openright,oneside]{article}

\usepackage[top=2.45cm, bottom=2.45cm, left=2.45cm, right=2.45cm]{geometry} 
\usepackage{amsmath}
\usepackage{amsthm}
\usepackage{amsfonts}
\usepackage{array}
\usepackage{graphicx} 
\usepackage{setspace} 
\usepackage{biblatex}
\usepackage{tikz}
\usepackage{hyperref}
\hypersetup{ colorlinks=true, linkcolor=orange, urlcolor=orange, citecolor = orange}
\usepackage{blindtext}
\usetikzlibrary{positioning}
\graphicspath{ {./images/} } 

\newtheorem{conjecture}{Conjecture}
\newtheorem{corollary}{Corollary}
\newtheorem{definition}{Definition}
\newtheorem{example}{Example}
\newtheorem{lemma}{Lemma}
\newtheorem{remark}{Remark}
\newtheorem{theorem}{Theorem}

\newcommand{\A}[1]{\mathcal{#1}} 
\newcommand{\length}{n} 
\newcommand{\str}[1]{#1} 
\newcommand{\window}{$w \times l$} 
\newcommand{\tori}{$v \times r$} 
\newcommand{\ringNum}{m} 
\newcommand{\ringSize}{p        } 
\newcommand{\deBS}[2]{deBS($#1$,$#2$)} 
\newcommand{\deBT}[5]{deBT($#2$,$#1$,$#4$,$#3$,$#5$)} 
\newcommand{\deBF}[4]{deBF($#1$,$#2$,$#3$,$#4$)} 
\newcommand{\AdeBS}[3]{AdeBS($#1$,$#2$,$#3$)} 
\newcommand{\AdeBG}[3]{AdeBG($#1$,$#2$,$#3$)} 

\begin{document}

\title{  Using alternating de Bruijn sequences to construct de Bruijn tori}
\author{
Matthew Kreitzer, Mihai Nica, Rajesh Pereira
\\ 
University of Guelph
}
\date{\today}
\maketitle

\tableofcontents

\begin{abstract}
A de Bruijn torus is the two dimensional generalization of a de Bruijn sequence. While methods exist to generate these tori, only a few methods of construction are known. We present a novel method to generate de Bruijn tori with rectangular windows by combining two variants of de Bruijn sequences. 

\end{abstract}

\section{Introduction}
De Bruijn tori (sometimes also called perfect maps) have applications in a number of different fields, including robot path planning~\cite{deBTRoboUse,robotPaper2,PMConstructor,PMClasses1,PMClass2}. Generating de Bruijn tori has been the subject of research since the 1980s~\cite{firstDeBT}. As of publication there exists some methods to generate de Bruijn tori~\cite{Horan2016Locating,deBTSquare,Lfill,2x3Cock1988,PMConstructor,PMClasses1}. All of these methods involve using certain types of de Bruijn sequences to construct de Bruijn tori. For instance, Horan and Stevens~\cite{Horan2016Locating} utilized a specialized de Bruijn sequence known as a quotient de Bruijn sequence and modular addition to produce de Bruijn tori. \par
We introduce two objects, alternating de Bruijn sequences and de Bruijn families, which can be used to construct de Bruijn tori. Alternating de Bruijn sequences are a generalization of a de Bruijn sequence which use two different alphabets, while a de Bruijn family is a collection of cyclic strings that have de Bruijn sequence properties when grouped together. \par
In Section~\ref{sec:defs}, we review definitions for de Bruijn sequences, de Bruijn tori, and formally introduce alternating de Bruijn sequences and de Bruijn families. The methods for generating alternating de Bruijn sequences and their corresponding alternating de Bruijn graphs are presented in Section~\ref{sec:AdeBSGen}. Section~\ref{sec:deBTGen} outlines a method for generating de Bruijn tori. It also discusses how our work is an abstraction of the approach presented in~\cite{2x3Cock1988}. Finally, Section~\ref{sec:examples} presents several examples of de Bruijn tori generated using alternating de Bruijn sequences and de Bruijn families. \par

\par

\section{Definitions}\label{sec:defs}
\subsection{A review of de Bruijn sequences and tori}
De Bruijn sequences were first discovered in 1894~\cite{OGdeBSpaper} and later rediscovered by de Bruijn in~\cite{deBruijnDef}. They are defined as follows:

\begin{definition}\label{def:deBS}
    A de Bruijn sequence of order $\length$ on the alphabet $\A{A}$ is a cyclical string of length ${|\A{A}|}^{\length}$ where every possible string of length $\length$ from the alphabet $\A{A}$ appears exactly once. We denote the set of all de Bruijn sequences with these parameters as \deBS{\A{A}}{\length}.
\end{definition}
\par
\begin{example} \label{ex:deBS}
The cyclic string $00010111$ is a member of \deBS{\A{A}=\{0,1\}}{n=3}. The length $3$ substrings of this cyclic string are $\{000,001,010,101,011,111,110,100\}$. These are precisely the set of all eight possible length $3$ words on the alphabet $\A{A}=\{0,1\}$.
\end{example}

\par
De Bruijn tori are the two dimensional generalization of de Bruijn sequences.
\par

\begin{definition}\cite[Definition 2.2]{Horan2016Locating}
    \label{def:deBT}
    A \textbf{de Bruijn torus} of window size \window \ with alphabet $\A{A}$ is an \tori\ array, where the rows/columns of the array are considered cyclically\footnote{The last row and column are adjacent to the first row and column respectively} and in which every rectangular window of size \window\;appears exactly once with $r,v,d,n,m \in \mathbb{N}$. 
     We denote the set of all de Bruijn tori with these parameters as \deBT{r}{v}{l}{w}{\A{A}}.
\end{definition}
\par
\begin{example}\label{ex:2DdeBT}
The following torus is a member of \deBT{4}{4}{2}{2}{\{0,1\}}: 
\[ \left( \begin{array}{cccc}
1 & 0 & 1 & 1  \\
1 & 0 & 0 & 0  \\
0 & 0 & 0 & 1  \\
1 & 1 & 0 & 1  \\
\end{array} \right) \] 
This torus is of size $ 4\times4 $ and contains each of the $16$ possible windows of size $ 2\times 2$ on the alphabet $\A{A} = \{0,1\}$, making it a de Bruijn torus. 
\end{example}
\par
\subsection{De Bruijn families}

A de Bruijn family is a collection of cyclic strings which, when considered together, have the properties of a de Bruijn sequence. 

\begin{definition}\label{def:deBF}
    A \textbf{de Bruijn family} of order $n$ on the alphabet $\A{A}$ is an unordered collection of cyclic strings, $\{\str{S_i}\}^{\ringNum}_{i=1}$, each of length $\ringSize$ with the following ``de Bruijn'' property: \\
    Every string of length $\length$ from the alphabet $\A{A}$ appears exactly once as a substring of some unique member $S_i$, $1\leq i \leq \ringNum$. We denote the set of all de Bruijn families with these parameters as \deBF{\A{A}}{\ringNum}{\ringSize}{\length}.
\label{deBF}
\end{definition} 


\begin{example}\label{ex:deBF1}
    The set of cyclic strings $\{ 0001, 1110\}$ are a member of \deBF{\A{A} = 
\{0,1\}}{\ringNum = 2}{\ringSize = 4}{\length = 3}. When broken into substrings of length three we obtain $\{000,001,010,100\}$ from the first string and $\{111,110,101,011\}$ from the second string. These are all eight possible three length substrings on the alphabet $\A{A} = \{0,1\}$.
\end{example}


\begin{example}\label{ex:deBF2}

The set of  $4$ cyclic strings 
$$
\{10100000,01011111,11100100,11011000\}.
$$
are a member of
\deBF{\A{A} = 
\{0,1\}}{\ringNum = 4}{\ringSize = 8}{\length = 5}
as they contain all possible length $5$ substrings on the alphabet $\A{A}\{0,1\}$
\end{example}

De Bruijn families are inspired by a problem in applied robotics~\cite{eBugs}. The authors of~\cite{eBugs} had constructed a set of disk shaped robots called eBugs. Each eBug had coloured lights evenly spaced among their circumference. The goal in~\cite{eBugs} was for an observer on the ground to be able to determine the identity and orientation of an eBug simply by looking at the lights on one side of the eBug.  In~\cite{eBugs}, it is shown that an optimal set of such coloured lights is equivalent to a de Bruijn family though this term is not used in this paper. Many properties of de Bruijn families can be found in \cite{eBugs}.  One conjecture from this paper is of particular interest as it suggests that de Bruijn families exist for a very wide family of parameters.

\begin{remark}\label{rem:deBF_size}
Note that by simply counting substrings, the condition $\ringNum \ringSize=|\A{A}|^\length$ is a necessary condition for the existence of a de Bruijn family. The authors of \cite{eBugs} have conjectured that  that this condition is also sufficient.
\end{remark}

\begin{conjecture}\cite[Conjecture 1]{eBugs}
\label{conj:ringNum}
 There exists a de Bruijn family in \deBF{\A{A}}{\ringNum}{\ringSize}{\length} if and only if $$\ringNum \ringSize=|\A{A}|^\length$$ and $\ringNum > \length$.
\end{conjecture}


In~\cite{eBugs}, some special cases of this conjecture were shown to be true.  We note a key example below.

\begin{theorem}\label{exist} \cite[Corollary 1]{eBugs}
Let $\A{A}$ be a finite alphabet and $\length$ be a positive integer.
    Then there exists a de Bruijn family in \deBF{\A{A}}{|\A{A}|}{|\A{A}|^{\length - 1}}{\length}.
\end{theorem}

Note that Example~\ref{ex:deBF1} is an  illustration of  Theorem \ref{exist}; in this case,  $\A{A}=\{ 0,1\}$ and $\length=3$.

\subsection{Alternating words and alternating de Bruijn sequences}\label{sec:AdeBS}

In order to define alternating de Bruijn sequences, we first define alternating words. 

\begin{definition}\label{def:altWord}
Let $\A{A}$ and $\A{B}$ be two finite alphabets. An \textbf{alternating word} on the alphabet pair $(\A{A},\A{B})$ is a string of length $2\length+1$ which alternates between elements of $\A{A}$ and elements of $\A{B}$ (beginning and ending with an element of $\A{A}$).  
The set of all $2\length + 1$ alternating words is the set $\A{A} \times \A{B} \times \A{A} \times \ldots \times \A{A} \times \A{B} \times \A{A}$ with $\length+1$ elements from $\A{A}$ and $\length$ elements from $\A{B}$. \end{definition}

\begin{remark}\label{rem:0length}
When we set $\length=0$ in the above definition, we find that alternating words of length one in ($\A{A},\A{B}$) are the words of length one in $\A{A}$. 
\end{remark}

\begin{example}\label{ex:altWord}
The set of strings
$$\{   a 0 a, a 0 b, b 0 a, b 0 b, a 1 a, a 1 b, b 1 a, b 1 b\},$$ 
are all 8 possible alternating word of length $3$ (i.e. $\length = 1$) for the alphabet pair $\A{A} = \{a,b\}$ and $\A{B} = \{0,1\}$. Note this is equivalent to $\A{A} \times \A{B} \times \A{A}$.
\end{example}

\begin{remark}\label{rem:sym}
The roles of the alphabet $\A{A}$ and $\A{B}$ in the definition of an alternating word are not symmetric.
\end{remark}

With this established, we can now define alternating de Bruijn sequences. 

\begin{definition}\label{def:AdeBS}
Let $\A{A}$ and $\A{B}$ be two finite alphabets. An \textbf{alternating de Bruijn sequence} of order $2\length + 1$ on the alphabet pair $(\A{A},\A{B})$, is a cyclic string which alternates between elements of $\A{A}$ and elements of $\A{B}$ and contains every alternating word on $(\A{A},\A{B})$ of length $2\length+1$ exactly once as a substring. We denote the set of all alternating de Bruijn sequences with these parameters as \AdeBS{\A{A}}{\A{B}}{2\length+1}.
\end{definition}

\begin{remark}\label{rem:AdeBSLen}
There are $\A{|A|}^{\length+1}\A{|B|}^{\length}$ alternating words of length $2\length + 1$. In an alternating de Bruijn sequence, every \emph{other} letter is an element of $\A{A}$ and begins exactly one of these alternating words. Therefore an alternating de Bruijn sequence of order $2\length + 1$ must have length $2\A{|A|}^{\length+1}\A{|B|}^{\length}$.
\end{remark}

\begin{example}\label{ex:AdeBS}
The string
$$ a \; 0 \; a \; 0 \; b  \; 0 \; b \; 0 \; a \; 1 \; a \; 1 \; b \; 1 \; b \; 1 $$
is a member of \AdeBS{\A{A}=\{a,b\}}{\A{B}=\{0,1\}}{3}. Note that every element of $\A{A}\times \A{B} \times \A{A}$ appears exactly once. In contrast, there is no restriction on  elements of $\A{B} \times \A{A} \times \A{B}$. For example, $0 b 0$ appears multiple times and  $0 b 1$ does not appear at all.
\end{example}

\begin{remark}\label{rem:AdeBs_vs_alt_word}
    Example~\ref{ex:AdeBS} shows that an alternating de Bruijn sequence is not strictly speaking an alternating word. Any alternating word has an odd number of characters, and both starts \emph{and} ends with an element of $\A{A}$. By contrast, the alternating de Bruijn sequence has an even number of characters and is a \emph{cyclic} word that simply alternates between alphabets. However, one can convert an alternating de Bruijn sequence of length $2v$ into an alternating word of length $2v+1$ by setting the last character to be equal to the first one. For example, by appending a letter $a$ onto the end, the alternating de Bruijn sequence in Example~\ref{ex:AdeBS} becomes the alternating word  $$ a \; 0 \; a \; 0 \; b  \; 0 \; b \; 0 \; a \; 1 \; a \; 1 \; b \; 1 \; b \; 1 \rightarrow a \; 0 \; a \; 0 \; b  \; 0 \; b \; 0 \; a \; 1 \; a \; 1 \; b \; 1 \; b \; 1 \; a$$
\end{remark}

In the construction of alternating de Bruijn sequences, one must consider the relationship between adjacent alternating words in an alternating de Bruijn sequence.  This relationship between alternating words of length $2\length + 1$ can be visualized in an alternating de Bruijn graph. 

\begin{definition}\label{def:AdeBG}

Let $\A{A}$ and $\A{B}$ be two finite alphabets.
For $n\geq 1$, \textbf{the alternating de Bruijn graph} of order $2 \length+1$ on the alphabet pair $(\A{A},\A{B})$ is the directed graph whose vertices are the alternating words of length $2\length+1$ on the alphabet pair $(\A{A},\A{B})$.
There is a directed edge from vertex $X = x_1 x_2...x_{2\length+1}$ to vertex $Y = y_1y_2...y_{2\length+1}$ if there exists an alternating word, $\str{S}$, of length $2\length-1$, such that $X = x_1x_2\str{S}$ and simultaneously  $Y=\str{S}y_{2\length}y_{2\length+1}$ (i.e. the suffix of $X$ is equal to the prefix of $Y$). We think of each edge in the graph as being labelled by the alternating word $x_1 x_2 S y_{2\length}y_{2\length+1}$ of length $2n+3$ (i.e. the edge between $X$ is $Y$ is the concatenation of $X$ and $Y$ writing their shared prefix/suffix only once).  
\par
For $\length=0$, the alternating de Bruijn graph is a multi-graph whose vertices are $\A{A}$. There is a directed edge from vertex $x \in \A{A}$ to vertex $y \in \A{A}$ for each element of the alphabet $\A{B}$. In this case, the set of edges from $x$ to $y$ is the set, $\{x \} \times \A{B} \times \{y\}$ which are alternating words of length $3$. We denote an alternating de Bruijn graph with these parameters as \AdeBG{\A{A}}{\A{B}}{2\length + 1}.
\par

\end{definition}

\begin{remark}\label{rem:AdeBGEdges}
Because of the shared prefix/suffix condition for edges, the alternating de Bruijn graph is designed so that any walk in the graph corresponds to a longer alternating word. The length of the word increases by $2$ for each edge of the path. See Lemma \ref{lem:paths_in_graph} for a precise statement of this idea.
\end{remark}

\begin{figure}\label{fig:AdebG(2,2,3)}
    \centering
    \begin{tikzpicture}[node distance={15mm}, thick, main/.style = {draw, circle}] 
        \node[main] (1) {$a0a$}; 
        \node[main] (2) [below = 3 cm of 1] {$a1a$}; 
        \node[main] (3) [above left = 3 cm of 1] {$a0b$}; 
        \node[main] (4) [below left = 3 cm of 2] {$a1b$}; 
        \node[main] (5) [above left = 2 cm of 2] {$b0a$}; 
        \node[main] (7) [above left = 3 cm of 4] {$b0b$}; 
        \node[main] (6) [above right = 2 cm of 7] {$b1a$}; 
        \node[main] (8) [above = 3 cm of 7] {$b1b$};
        
        \draw[->] (1) to [out=35,in=325,looseness=4] (1) node [right=0.9cm, fill=orange] {1};
        \draw [->,out=300,in=60,looseness=1] (1.300) to node[midway, fill = orange]{2}  (2.60); 
        \draw[->] (1) -- (3) node [midway, fill=orange] {23};
        \draw [->,out=305,in=20,looseness=1.75] (1.305) to node[midway, fill = orange]{5}  (4.20); 

        \draw[->] (2) -- (1) node [midway, fill=orange] {4};
        \draw[->] (2) to [out=35,in=325,looseness=4] (2) node [right = 0.8cm, fill=orange] {3}; 
        \draw[->] (2) -- (3) node [midway, fill=orange] {27};
        \draw[->] (2) -- (4) node [midway, fill=orange] {18}; 

        \draw [->,out=300,in=60,looseness=1.5] (3.300) to node[midway, fill = orange]{12}  (5.60); 
        \draw [->,out=230,in=110,looseness=1] (3.230) to node[midway, fill = orange]{8}  (6.110); 
        \draw[->] (3) -- (7) node [midway, fill=orange] {28};
        \draw[->] (3) -- (8) node [midway, fill=orange] {24}; 

        \draw[->] (4) -- (5) node [midway, fill=orange] {10};
        \draw[->] (4) -- (6) node [midway, fill=orange] {6};
        \draw[->] (4) -- (7) node [midway, fill=orange] {14};
        \draw [->,out=160,in=230,looseness=1.75] (4.160) to node[midway, fill = orange]{19}  (8.230); 
        
        \draw [->,out=20,in=245,looseness=1.6] (5.20) to node[midway, fill = orange]{22}  (1.245);
        \draw[->] (5) -- (2) node [midway, fill=orange] {17};
        \draw[->] (5) -- (3) node [midway, fill=orange] {11};
        \draw [->,out=280,in=60,looseness=1] (5.280) to node[midway, fill = orange]{13}  (4.60);

        \draw [->,out=60,in=180,looseness=1] (6.60) to node[midway, fill = orange]{32}  (1.180); 
        \draw[->] (6) -- (2) node [midway, fill=orange] {26};
        \draw[->] (6) -- (3) node [midway, fill=orange] {7};
        \draw [->,out=260,in=120,looseness=1] (6.260) to node[midway, fill = orange]{9}  (4.120);

        \draw[->] (7) -- (5) node [midway, fill=orange] {16};
        \draw[->] (7) -- (6) node [midway, fill=orange] {31};
        \draw[->] (7) to [out=140,in=220,looseness=4] (7) node [left=0.6cm, fill=orange] {15}; 
        \draw[->] (7) -- (8) node [midway, fill=orange] {29};  

        \draw [->,out=10,in=130,looseness=1.5] (8.10) to node[midway, fill = orange]{21}  (5.130);
        \draw[->,out=300,in=160,looseness=1]  (8.300) to node [midway, fill=orange] {25} (6.160) ;
        \draw [->,out=240,in=120,looseness=1] (8.240) to node[midway, fill = orange]{30}  (7.120);
        \draw[->](8) to [out=215,in=130,looseness=4] (8) node [left=0.8cm, fill=orange] {20}; 

    \end{tikzpicture} 
    \caption{The alternating de Bruijn graph for \AdeBG{\A{A} = \{a,b\}}{\A{B} = \{0,1\} }{3}. Here two vertices have a directed edge if they produce an alternating word of length $5$ when `glued together'. The numbering of the edges show an Eulerian cycle starting at $a0a$.}
\end{figure}
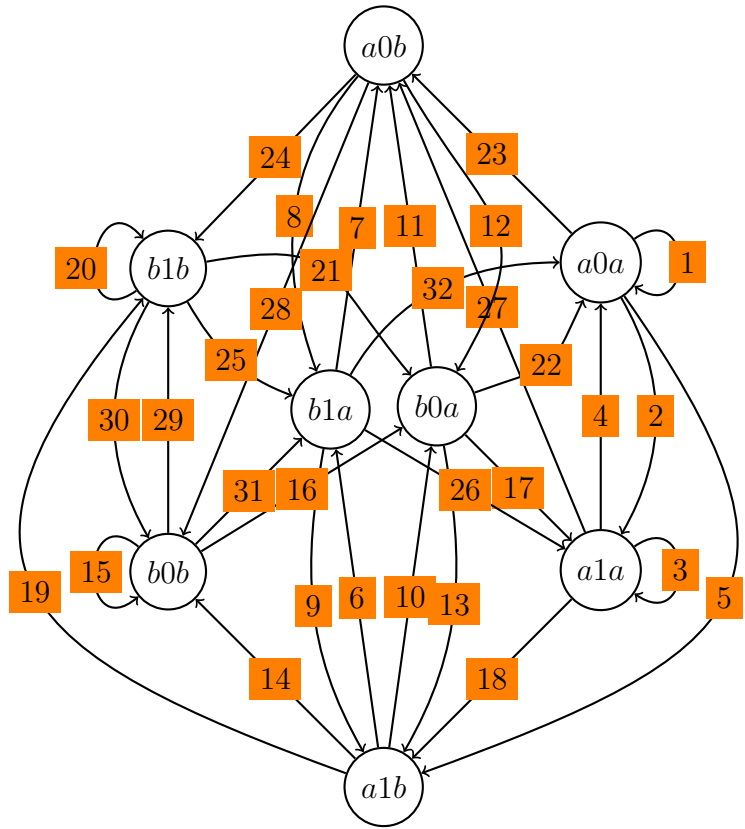
\section{Construction of alternating de Bruijn sequences}\label{sec:AdeBSGen}
The method we develop to generate alternating de Bruijn sequences is inspired by the classical method of using Eulerian cycles to generate de Bruijn sequences  from de Bruijn graphs. An Eulerian cycle is a walk that visits every edge once and once only (but can visit vertices more than once). An Eulerian digraph is a digraph with an Eulerian cycle.  To start, we state Euler's Theorem, then use it prove the existence of alternating de Bruijn sequences. We note that this argument is very similar to (and inspired by) the graph theoretic proof of the existence of de Bruijn sequences.
\begin{theorem}\cite[Theorem 4.6]{chartrand2010graphs} 
    Let $G$ be a connected digraph.  Then $G$ is an Eulerian digraph if and only if for all vertices, $v$, of $G$, the indegree of $v$  is equal to the outdegree of $v$.
\end{theorem}
\begin{remark}\label{rem:ECycleExist}
    Any vertex of an alternating de Bruijn graph of any order on the alphabets $\A{A}$ and $\A{B}$ must have both indegree and outdegree equal to $\vert \A{A} \vert \vert \A{B}\vert$. Hence by Euler's theorem, this graph must have Eulerian cycles.
\end{remark}
\begin{lemma} \label{lem:paths_in_graph}
Any path $\pi$ in \AdeBG{\A{A}}{\A{B}}{2\length+1} with $k$ edges corresponds to an alternating word $W$ of length $2(k-1) + 2\length + 3$. The $k$ edges in $\pi$ correspond to the $k$ alternating words of length $2\length + 3$ appearing as substrings of $W$.
\end{lemma}
\begin{proof}
By Definition \ref{def:AdeBG}, individual edges of the alternating de Bruijn graph represent the alternating words of length $2\length + 3$. Thus $\pi$ can be thought of as a sequence of $k$ alternating words. Moreover, by the definition of the directed edges, the words along such a path are ``compatible'' in the sense that the suffix of one word is the prefix of the next word; exactly $2$ new characters (one from $\A{A}$ and one from $\A{B}$) are added with each new word. Because of this compatibility, the words along a path with $k$ edges can be ``glued together'' and viewed as the $k$ substrings (each of length $2\length + 3$) of an alternating word of total length  $2(k-1) + 2\length + 3$.
\end{proof}

\begin{theorem} \label{th:de_Bruijn_}
    There is a canonical bijection between Eulerian cycles on the alternating de Bruijn graphs of order $2\length + 1$, \AdeBG{\A{A}}{\A{B}}{2\length+1}, and the set of alternating de Bruijn sequences of window size $2\length + 3$ (namely \AdeBS{\A{A}}{\A{B}}{2\length+3}).
\end{theorem}
\par
\begin{proof}
Specializing Lemma~\ref{lem:paths_in_graph} to cycles, we see that any \emph{cycle} through the graph will correspond to a \emph{cyclic} alternating word (with edges in the cycle corresponding to length $2\length + 3$ substrings of the cyclic word). An \emph{Eulerian} cycle visits each edge exactly once; this corresponds precisely to having each alternating word of length $2\length + 3$ appearing exactly once as a substring of the cyclic word. This is precisely the de Bruijn property, as desired.
\end{proof}

\begin{corollary} \label{cor:adeBS_exists}
For any alphabets $\A{A}$ and $\A{B}$, there always exists at least one alternating de Bruijn sequence on $\A{A}$ and $\A{B}$ of any window size.
\end{corollary}
\begin{proof}
By Euler's theorem and Remark \ref{rem:ECycleExist}, we know that Eulerian cycles exist on the alternating de Bruijn graph. By Theorem \ref{th:de_Bruijn_}, these Eulerian cycles correspond exactly to alternating de Bruijn sequences.
\end{proof}

\begin{example}\label{ex:eCycle}
Consider the following cycle on the graph \AdeBG{\A{A}=\{a,b\}}{\A{B}=\{0,1\}}{3}:
    \begin{gather}
    \nonumber 
    a0a \rightarrow a0a \rightarrow a1a \rightarrow a1a \rightarrow a0a \rightarrow a1b \rightarrow b1a \rightarrow a0b \rightarrow b1a \rightarrow a1b \rightarrow b0a \rightarrow a0b \rightarrow b0a \rightarrow  ... \\
    \nonumber 
    a1b \rightarrow b0b \rightarrow b0b \rightarrow b0a \rightarrow a1a \rightarrow a1b \rightarrow b1b \rightarrow b1b \rightarrow b0a \rightarrow a0a \rightarrow a0b \rightarrow b1b \rightarrow b1a \rightarrow ... \\
     a1a \rightarrow  a0b \rightarrow b0b \rightarrow b1b \rightarrow b0b \rightarrow b1a \rightarrow a0a .
        \label{eq:ePath}
    \end{gather}
Figure \ref{fig:AdebG(2,2,3)} illustrates this cycle where the numbered edges represent the order of edges taken by the path. This cycle is an Eulerian cycle because every edge is traversed exactly once. By ``gluing'' the edges along the cycle together, we obtain the cyclic word,
$$ a0a0a1a1a0a1b1a0b1a1b0a0b0a1b0b0b0a1a1b1b1b0a0a0b1b1a1a0b0b1b0b1. $$
This is a  member of \AdeBS{\A{A}=\{a,b\}}{\A{B}=\{0,1\}}{5} since each alternating word of length $5$ appears exactly once.
\end{example}

\par
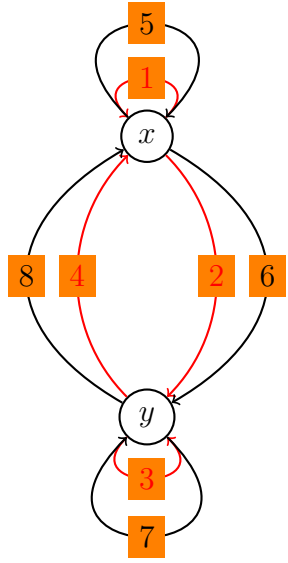
\begin{figure}\label{fig:AdebG(2,2,0)}
    \centering
    \begin{tikzpicture}[node distance={15mm}, thick, main/.style = {draw, circle}] 
        \node[main] (1) {$x$}; 
        \node[main] (2) [below = 3 cm of 1] {$y$}; 
        
        \draw[->,out=45,in=135,looseness=5, color = red] (1) to node[midway, fill = orange]{1} (1);
        \draw[<-,out=45,in=135,looseness=12] (1) to node[midway, fill = orange]{5} (1);

        \draw[->,out=315,in=45,looseness=1, color = red] (1) to node[midway, fill = orange]{2} (2);
        \draw[->,out=330,in=30,looseness=1.5]  (1) to node[midway, fill = orange]{6} (2);
        
        \draw[->,out=135,in=225,looseness=1, color = red] (2) to node[midway, fill = orange]{4} (1);
        \draw[->,out=150,in=210,looseness=1.5] (2) to node[midway, fill = orange]{8} (1); 
        
        \draw[->,out=225,in=315,looseness=5, color = red] (2) to node[midway, fill = orange]{3} (2);
        \draw[<-,out=225,in=315,looseness=12] (2) to node[midway, fill = orange]{7} (2); 
        
    \end{tikzpicture} 
    \caption{This is the alternating de Bruijn graph corresponding to \AdeBG{\A{A} = \{x,y\} }{\A{B} = \{a,b\} }{1}. The character $a$ is represented by the red edges, while $b$ is represented by the black edges. We number the Eulerian cycle that generates the alternating de Bruijn sequence in Example~\ref{ex:AdeBS}.}
\end{figure}

\par

We end this section with two observations about alternating de Bruijn graphs that further explain their connections to de Bruijn graphs.

\begin{definition}
    Let $G_1=(V_1,E_1)$ and $G_2=(V_2,E_2)$ be two directed graphs, then the tensor product of $G_1$ and $G_2$ (denoted $G_1\otimes G_2$) is the graph with vertex set $V_1\times V_2$ where $((u_1,u_2),(v_1,v_2))$ is an edge whenever both $(u_1,v_1)\in E_1$ and $(u_2,v_2)\in E_2$.
\end{definition}

The tensor product of graphs is sometimes called the Kronecker product.  The tensor product has been used to study de Bruijn graphs in \cite{shibata2000isomorphic}, though we will not use any of those results here.  Rather we observe that the alternating de Bruijn graph is the tensor product of two de Bruijn graphs.

\begin{remark}\label{rem:AdeBGTensor}
    There is a graph isomorphism from the \AdeBG{\A{A}}{\A{B}}{2\length+1} to the tensor product between the de Bruijn graph of order $\length+1$ on alphabet $\A{A}$ and the de Bruijn graph of order $\length$ on alphabet $\A{B}$. This can be seen as any edge between two vertices of the alternating de Bruijn graphs, i.e the edge between $v_1 = x_1y_1x_2y_2 \dots x_{\length}y_{\length}x_{\length+1}$ and $v_2 = x_2y_2 \dots x_{\length+1}y_{\length+1}x_{\length+2}$, has a corresponding edge in the tensor product, i.e the edge between $w_1 = (x_1x_2\dots x_{\length}x_{\length+1}, y_1y_2\dots y_{\length}y_{\length+1})$ and $w_2 = (x_2\dots x_{\length+1}x_{\length+2}, y_2\dots y_{\length+1}y_{\length+2})$.
\end{remark}

Our second observation is that alternating de Bruijn graphs have the same line digraph property as de Bruijn graphs.

\begin{definition} Let $G=(V,E)$ be a directed graph.  Then the line digraph of $G$ (denoted $L(G)$) is a directed graph with vertex set $E$ with a directed edge from $e_1$ to $e_2$ if $e_1$ ends at the same vertex that $e_2$ emanates from. 
    \end{definition}

\begin{remark}  The line digraph of an alternating de Bruijn graph of order $2n+1$ is the alternating de Bruijn graph of order $2n+3$.  It follows from this that all alternating De Bruijn graphs are Hamiltonian (i.e has a cycle that visits every vertex once and once only).
\end{remark}

\section{Construction of de Bruijn Tori}\label{sec:deBTGen}
In this section, we describe a method that uses a de Bruijn family and an alternating de Bruijn sequence to generate a de Bruijn torus on a $v \times r$ array, with rectangular window size \window\; and some given alphabet $\A{O}$. To start, we define a rotation on a string as follows:
\begin{definition}\label{def:rotation}
    Let $x$ be an integer.  We denote by $\pi_x$ the forward rotation of a cyclic string by $x$ characters. Precisely, this is defined as follows. For the string $\str{S} = s_0s_1...s_{r-1}$ we have
    \begin{equation}
        \pi_x(\str{S}) = \pi_x(s_0s_1...s_{r-1}) = s_{x}s_{x+1}...s_{r-1+x},
    \end{equation}    
    where all the subscripts are understood to be mod $r$. For example $\pi_{-1}(1000) = 0001$ and $\pi_{-2}(1000) = 0010$.
\end{definition}

We use the rotations $\pi_x$ to define an operation which converts an alternating de Bruijn sequence into a torus by stacking and rotating strings. We denote this construction by $\sigma$ and define this precisely as follows.
\begin{definition}\label{def:deBTOP}[Alternating Words $\to$ Array construction]
Let $\A{S}=\{S_1,\ldots,S_k\}$ be a collection of cyclic strings, each of length $r$, on the alphabet $\A{O}$. Let $\Pi = \{\pi_0, \pi_1,\ldots, \pi_{r-1}\}$ be the set of rotations on cyclic strings of length $r$. Let $\str{D}$ be an alternating word the alphabet pair $(\A{S},\Pi)$ of length $2v+1$, i.e. 
$$\str{D} = \str{S}_{m_1}\pi_{r_1}...\str{S}_{m_v}\pi_{r_v}\str{S}_{m_{v+1}},$$
for some choice of indices $0 \leq m_1,\ldots,m_{v+1} \leq k$ and $0 \leq r_1,\ldots,r_v \leq r$. We define the $ (v+1)\times r$ array with alphabet $\A{O}$, which we denote by $\sigma(D)$,  as follows
    \begin{equation} \label{eq:defConst}
        \centering
        \sigma(\str{D})=\left[\begin{array}{c}
        \str{S}_{m_{1}}\\
        \pi_{r_{1}}(\str{S}_{m_{2}})\\
        \pi_{r_{2}+r_{1}}(\str{S}_{m_{3}})\\
        \vdots\\
        \pi_{r_{v}+\ldots+r_{2}+r_{1}}(\str{S}_{m_{v+1}})
        \end{array}\right].
    \end{equation}
Notice that the rows of the array $\sigma(D)$ consist of rotated copies of the strings $S$.
\end{definition}

\begin{lemma}[Condition for the array $\sigma(D)$ to be a torus]\label{lem:array_is_torus}
As in the setup of Definition~\ref{def:deBTOP}, let $\A{S}=\{S_1,\ldots,S_k\}$ be a collection of cyclic strings each of length $r$ on the alphabet $\A{O}$, let $\Pi = \{\pi_0, \pi_1,\ldots, \pi_{r-1}\}$ be the rotations on these cyclic strings, and let $\str{D} = \str{S}_{m_1}\pi_{r_1}...\str{S}_{m_v}\pi_{r_v}\str{S}_{m_{v+1}}$.  \par
Suppose that $m_1 = m_{v+1}$ (i.e. the first and last cyclic string  of $D$ are equal) and that $r_v+\ldots+r_1 \equiv 0 \text{ mod } r$. Then the first and last row of $\sigma(D)$ are equal, and the $(v+1)\times r$ array $\sigma(D)$ can be thought of as an $v \times r$ \emph{torus}.
\end{lemma}
\begin{proof}
Since $r_v+\ldots+r_1 \equiv 0 \text{ mod } r$, we see that $\pi_{r_{v}+\ldots+r_{2}+r_{1}} = \pi_0 $. Furthermore since $m_1 = m_{v+1}$, we have $\pi_{r_{v}+\ldots+r_{2}+r_{1}} S_{m_{v+1}} =\pi_0 S_{m_1} = S_{m_1}$. Thus the first row is equal the last row and we can interpret $\sigma(D)$ a torus of size $v \times r$ by identifying these rows. 
\end{proof}

\begin{lemma}\label{lem:deBTWraps}
As in the setup of Definition~\ref{def:deBTOP}, let $\A{S}=\{S_1,\ldots,S_k\}$ be a collection of cyclic strings each of length $r$ on the alphabet $\A{O}$ and let $\Pi = \{\pi_0, \pi_1,\ldots, \pi_{r-1}\}$ be the rotations on these cyclic strings.  

Let $D$ be an alternating de Bruijn sequence of order $2\length +1$ and length $2v$ on the alphabets $\A{S}$ and $\Pi$. Thinking of $D$ as an alternating word of length $2v+1$ as in Remark \ref{rem:AdeBs_vs_alt_word} by setting the first/last letter to be the same, we can construct the $(v+1) \times r$ array $\sigma(D)$ as in Definition~\ref{def:deBTOP}.

The array $\sigma(D)$ is actually an $v \times r$ \emph{torus} as long as \emph{at least} one of the following three conditions are met:
\begin{itemize}
    \item $|\A{S}|$ is even
    \item $|\Pi|$ is odd
    \item $n \geq 2$
\end{itemize}
\end{lemma}
\begin{proof} By Lemma \ref{lem:array_is_torus}, we only require that ${r_{v}+\ldots+r_{2}+r_{1}} \equiv 0 \text{ mod } r $, where $r= |\Pi|$. Note the number of rotation indices is $v= |\A{S}|^{n+1} r^n$, because this is the number of alternating words of length $2\length+1$ on the alphabets $(\A{S},\Pi)$. The \emph{average rotation} is simply $\frac{1}{2}(r-1)$, as by symmetry each of the rotations $0,\dots,|\Pi|-1$ occurs an equal number of times in an alternating de Bruijn sequence. Thus the sum of rotation indices is 
\begin{equation}
    {r_{v}+\ldots+r_{2}+r_{1}} = |\A{S}|^{n+1}r^{n} \frac{1}{2}(r-1) = r^{n-1} |\A{S}|^{n+1}  \frac{r(r-1)}{2} 
\end{equation}
Since $\frac{r(r-1)}{2}$ is always an integer, if $n\geq 2$, then this sum clearly divides $r$ from the factor $r^{n-1}$. When $n=1$, the sum is simply $\frac{1}{2} |\A{S}|^{2} r(r-1)$, which is divisible by $r$ when $r-1$ is even or $|\A{S}|$ is even. These are exactly the desired conditions of the lemma. 
\end{proof}

We now use a specialized de Bruijn family, and alternating de Bruijn sequences whose second alphabet is a group of rotations to generate a de Bruijn torus.

\begin{theorem}\label{th:deBTExist}
    Let $\A{S}=\{\str{S}_1,\dots,\str{S}_m\}$ be a de Bruijn family of order $\ell$ on some alphabet $\A{O}$ consisting of $m$ cyclic strings each of length $r$. Let $D$ be an alternating de Bruijn sequence of order $2n+1$ on the alphabet pair $\A{S} = \{\str{S}_1, ... \str{S}_m\}$ and $\Pi = \{\pi_0, ... \pi_{r-1}\}$\footnote{$\Pi$ being the rotations on strings of length $r$}. Suppose that one of the three conditions of Lemma~\ref{lem:deBTWraps} is satisfied so that $\sigma(D)$ is a torus.
    
    Let $w = n+1$ and let $v$ be half the length of $D$, $v = \frac{|D|}{2}$. Then the torus $\sigma(D)$ is in fact a \emph{de Bruijn torus} on a \tori\; array with window size \window\; and alphabet $\A{O}$.
\end{theorem}


\begin{proof}
    To prove this, we will suppose we are given a window $W$, which is an arbitrary \window\; array from the alphabet $\A{O}$. We will first show how to map the window $W$ into a specific alternating word $\str{D}_W$ on the alphabets $\mathcal{S},\Pi$. By the alternating de Bruijn property, $\str{D}_W$ exists as a unique substring of $\str{D}$. We will then show that the existence/uniqueness of $\str{D}_W$ a a substring of $D$ directly leads to the existence/uniqueness of the window $W$ as a subarray of $\sigma(D)$.\par
    We start by indexing $\str{D}$ as follows, with $\ringNum_i,r_i, 1 \leq i \leq v$:
    \begin{equation}
        \str{D} = \str{S_{\ringNum_1}}\pi_{r_1}\str{S}_{\ringNum_2}\pi_{r_2}\ldots\str{S}_{\ringNum_v}\pi_{r_v},
    \end{equation}
    so that $\str{S}_{m_i}$ represents the i-th element of the alternating de Bruijn sequences. We now construct the torus $\sigma(D)$, as in Definition \ref{def:deBTOP}, as follows:
    \begin{equation} \label{eq:deBTConst}
        \centering
        \sigma(D)= \left[\begin{array}{c}
        \str{S}_{m_{1}}\\
        \pi_{r_{1}}(\str{S}_{m_{2}})\\
        \pi_{r_{2}+r_{1}}(\str{S}_{m_{3}})\\
        \vdots\\
        \pi_{r_{v}+\ldots+r_{2}+r_{1}}(\str{S}_{m_{v}})
        \end{array}\right]
    \end{equation}\par
    
    Notice each row are the members of $\{S_j\}^m_{j=1}$ each with the rotations $\pi_{r_i}$ applied \emph{cumulatively} to the strings in the de Bruijn family.
    \par
    To show the constructed array is a de Bruijn torus, we show that each window of size \window\; appears in a unique location. Let $W$ be some array of size \window\; with alphabet $\A{O}$ with rows denoted $W_i$ as follows
    \begin{equation}\label{eq:deBTChunk}
        \centering
        W=\left[\begin{array}{c}
        W_{1}\\
        W_{2}\\
        \vdots\\
        W_{w}
        \end{array}\right]    
    \end{equation}\par
    Since each $W_i$ is a string of length $l$, the de Bruijn family property of $\A{S}$ forces $W_i$ to appear exactly once in some element of  $\{\str{S}_j\}^m_{j=1}$ at some particular location within that string. Denote which string $W_i$ appears in by $S(W_i)$ and denote  $W_i$'s location within $S(W_i)$ by $r(W_i)$ (e.g. if we had $S(W_i) = S_1$ and $r(W_i) = 0$, this means that that $W_i$ appears as a substring of $S_1$, starting at position $0$, i.e. the first $l$ characters $S_1$). 
    \par
    We now denote the difference in successive locations, $\Delta r(W_i)$ to be $\Delta r(W_i) = r(W_i) - r(W_{i+1})$. Now consider the alternating word of length $2n+1$, $\str{D_W}$, on the alphabets $\A{A}$ and $\A{B}$ defined as follows
    \begin{equation}\label{eq:Dw}
        \str{D}_W := S(W_1)\pi_{\Delta r(W_1)} ... \pi_{\Delta r(W_{w-1})}S(W_{w})
    \end{equation}\par
    By the properties of alternating de Bruijn sequences, $\str{D}_W$ appears exactly once as a substring of $\str{D}$.  We will show that this forces $W$ to appear exactly once as a subarray of $\sigma(D)$.
    
    To see that $\str{D}_W$ exists as a substring in $D$ implies $W$ exists as a subarray of  $\sigma(D)$, consider as follows. Consider the $w \times r$ array $\sigma(D_W)$. Now rotate all the rows of this array by $\pi_{-r(W_1)}$, and denote the resulting $w \times r$ array by $\pi_{-r(W_1)} \sigma(D_W)$. The differences $\Delta r_i$ that appear in the alternating word $D_W$ are summed in the cumulative sum of rotations in the construction of $\sigma$ from Definition~\ref{def:deBTOP}, and we see that the array is actually
    \begin{equation}
        \label{eq:sigma_DW}
        \pi_{-r(W_1)}(\sigma(D_W)) = 
        \left[\begin{array}{c}
        \pi_{-r(W_1)}S(W_{1})\\
        \pi_{-r(W_2)}S(W_{2})\\
        \vdots\\
        \pi_{-r(W_{w})}S(W_{w})
        \end{array}\right]    
    \end{equation}
    By definition of $r(W_i)$, applying the rotation $\pi_{-r(W_i)}$ to the string $S(W_i)$ rotates the string $S(W_i)$ such that the first $l$ characters in the rotated string $\pi_{-r(W_i)}S(W_i)$ is simply the string $W_i$. Therefore, from \eqref{eq:sigma_DW}, we see that the first $\ell$ characters of each row is the string $W$. Since this is true for every row, actually the first $\ell$ columns of $\pi_{-r(W_1)}(\sigma(D_W))$ are precisely the window $W$. Therefore $W$ is a subarray of $\sigma(D_W)$. By the construction of $\sigma(D)$, we see that $\sigma(D_W)$ is itself a subarray of $\sigma(D)$. Therefore, the window $W$ is a subarray of $\sigma(D)$ as desired. 
    
    \par 
    To see that $\str{D}_W$ is unique as a substring of $D$ implies $W$ is unique as a subarray of  $\sigma(D)$, consider as follows. By the de Bruijn family property of $\A{S}$, there exists a unique member of $\A{S}$ in which $W_i$ appears, namely $S(W_i)$. Therefore, the only way the window $W$ can appear in $\sigma(D)$ is when the strings $S(W_1),\ldots, S(W_w)$ appear consecutively. Moreover, when strings $S(W_1)\dots S(W_w)$ are stacked above one another here, the rotations $\delta r(W_i)$ must be applied between rows for the window $W$ to appear. (Otherwise, the appearance of $W_i$ in row $i$ would not be directly above the appearance of $W_{i+1}$ in row $i+1$; the rotation $\Delta r(W_i)$ is needed to line them up). This arrangement of the rows $S(W_1),\ldots,S(W_w)$ and rotations $\Delta r(W_1),\ldots,\Delta r(W_w)$ is exactly how we constructed the alternating word $D_W$. Therefore, by our construction, any appearance of the window $W$ as a subarray of $\sigma(D)$ corresponds to finding $D_W$ as a substring of $D$. By the uniqueness of $D_W$ as a substring $D$, we therefore conclude the uniqueness of $W$ as a subarray of $\sigma(D)$.

    Since $W$ was arbitrary, we have shown that every window $W$ appears exactly once as a subarray of $\sigma(D)$. Therefore $\sigma(D)$ is a de Bruijn torus, as desired.
        
\end{proof}
    
\begin{corollary}\label{cor:conditionVersion}
In order to construct an $v \times r$ de Bruijn torus with window of size \window\; (where $\ell<r$ and $w<v$) on some alphabet $\A{O}$ by the method described in Theorem~\ref{th:deBTExist}, the following three conditions must hold.
\begin{enumerate}
\item Alphabet size condition: It is necessary that
\begin{equation}\label{eq:rv_Olw}
rv = |\mathcal{O}|^{\ell  w}.
\end{equation}  
\item De Bruijn family size condition: It is necessary that there is an integer $m$ so that
\begin{equation}\label{eq:v_rwm} 
v = r^{w-1} m^w.
\end{equation}
\item De Bruijn family existence condition: There exists a de Bruijn family $\A{S} \in $\deBF{\A{O}}{\ringNum}{r}{\ell} consisting of $m$ strings on the alphabet $\A{O}$ which are each length $r$, and which contain all possible substrings of length $\ell$ exactly once.
\end{enumerate}
\end{corollary}
\begin{proof}
Condition \eqref{eq:rv_Olw} simply ensures that the number of \window\; subarrays of is the same as the number of windows of size \window\; from the alphabet $\A{O}$. Condition \eqref{eq:v_rwm} follows by combining \eqref{eq:rv_Olw} with the necessary condition for de Bruijn families $rm = \vert \A{O} \vert^\ell$ from Remark \ref{rem:deBF_size}. Once these conditions are met, we have all the ingredients for the application of Theorem \ref{th:deBTExist} except for the the alternating De Bruijn sequence. But alternating De Bruijn sequences always exist (when the size conditions are not violated) by Theorem \ref{th:de_Bruijn_}, so in principle one can be found and the construction carried out (In practice, one may find it by finding an Eulerian cycle in the appropriate alternating de Bruijn graph). 
\end{proof}

\begin{remark}

If Conjecture \ref{conj:ringNum} is true, then de Bruijn families will also always exist as long as the de Bruijn family size condition is satisfied, and in that case Theorem \ref{th:deBTExist} will show the existence of a de Bruijn tori whenever conditions \eqref{eq:rv_Olw} and \eqref{eq:v_rwm} are both satisfied.
\end{remark}

\begin{remark}
    Notice both $v > w$ and $r > \ell$ is required for the de Bruijn families and alternating de Bruijn sequences to exist. 
\end{remark}

\begin{remark}\label{rem:squareDeBT}
    Since $r = |\Pi|$, $m = |\A{S}|$ and $w = n+1$ Equation~\eqref{eq:v_rwm} can be rewritten as:      \begin{equation}
        v = |\A{S}|^{\length+1}|\Pi|^{\length}
    \end{equation}    
\end{remark}

Since Theorem \ref{exist} gives the existence of a class of de Bruijn families, it can now be used to show the existence of a family of de Bruijn tori.  We note this in the following corollary.

\begin{corollary}
  Let $\A{O}$ be a finite alphabet and let $l,w\in\mathbb{N}$, then there exists an $|\A{O}|^{l-1}$ by $|\A{O}|^{lw-l+1}$ de Bruijn torus with a window size of \window \ on the alphabet $\A{O}$. 
\end{corollary}

An example of the de Bruijn torus in the case $\A{O}=\{0,1\}$ and $\ell=w = 3$ is given in Example \ref{ex:3x3fixed}.

For purposes of comparison, we note that the method in \cite{2x3Cock1988} gives constructs the following Be Bruijn tori.

\begin{theorem}\cite{2x3Cock1988}\label{cock}
   Let $\A{O}$ be a finite alphabet and let $l,w\in\mathbb{N}$, then there exists an $|\A{O}|^{l}$ by $|\A{O}|^{lw-l}$ de Bruijn torus with a window size of \window \ on the alphabet $\A{O}$.  
    
\end{theorem}

 We discuss Theorem \ref{cock} further in the next subsection.

\subsection{The special case $\ringNum = 1$.}

Setting $\ringNum = 1$ turns the required de Bruijn family into a de Bruijn sequence. Using Corollary~\ref{cor:conditionVersion}, $\ringNum = 1$ shows $r = |\A{O}|^l$. Furthermore we obtain the relationship, $v=|\A{O}|^{\ell(w-1)}$.
Construction of the required alternating de Bruijn sequence becomes trivial as its first alphabet has only one member. 
\begin{lemma}\label{lem:m=1}
Given a de Bruijn sequence with alphabet $\A{B}$ and order $\length$, it is trivial to construct an alternating de Bruijn sequence on the alphabet pair $\{\A{A}, \A{B}\}$ of order $2\length+1$ when $|\A{A}| = 1$.
\end{lemma}
\begin{proof}
    Take any de Bruijn sequence of alphabet $\A{B}$ and order $\length$, i.e  $\str{S} = s_1s_2...s_l$, and interlace the only member of $\A{A}$ to construct $\str{A} = as_1as_2...as_l$. The resulting string $\str{A}$ is cyclic. Furthermore it contains every alternating word of length $2\length + 1$, thus $\str{A}$ is an alternating de Bruijn sequence. 
\end{proof}

This is an interesting result which has already been discovered in 1988 by John C. Cock~\cite{2x3Cock1988}. Our construction method can be viewed as a generalization of his work. The following example reproduces the $c = 2, m = 2, n = 3$ \footnote{Here we use the variable convention in \cite{2x3Cock1988}.} example in his paper. Even for the $\ringNum = 1$ case, our method gives some new cases where Cock's method can work as the conditions in Lemma~\ref{lem:deBTWraps} are more broad than those in \cite{2x3Cock1988} . 

\begin{example}\label{ex:1988deBT}
    The following is a de Bruijn torus on $16 \times 4$ array with a rectangular window size $3 \times 2$ and alphabet $\A{O} = \{0,1\}$. (NOTE: We have presented the transpose of the de Bruijn torus here so that it fits more nicely on the page)
    \begin{equation}\label{eq:1988deBT}
        \centering
        \sigma(D)^T = \left[\begin{array}{cccccccccccccccc}
            0&0&0&0&0&1&1&1&1&0&1&1&1&0&1&0\\
            0&0&0&1&1&0&0&1&0&0&1&0&1&0&1&1\\
            1&1&1&1&1&0&0&0&0&1&0&0&0&1&0&1\\
            1&1&1&0&0&1&1&0&1&1&0&1&0&1&0&0\\
        \end{array}\right]    
    \end{equation}
  
    This de Bruijn torus was constructed using the construction of Theorem \ref{th:deBTExist} and the one-member de Bruijn family
    $\A{A} = \{0011\}$
    and alternating de Bruijn sequence, $\str{D}$, of order $2n+1 = 5$ on alphabet pair $(\A{A} = \{a=0011\},\A{B} = \{\pi_0,\pi_1,\pi_2,\pi_3\})$: 
    \setlength\arraycolsep{2pt}
    \begin{equation}
        D = a\;\pi_0\; a\;\pi_0\; a\;\pi_1\; a\;\pi_0\; a\;\pi_2\; a\;\pi_0\; a\;\pi_3\; a\;\pi_1\; a\;\pi_1\; a\;\pi_2\; a\;\pi_1\; a\;\pi_3\; a\;\pi_2\; a\;\pi_2\; a\;\pi_3\; a\;\pi_3 
    \end{equation}
We now construct the torus as $\sigma(\str{D})^T$ as in Definition \ref{def:deBTOP} to obtain\footnote{i.e. we apply the  cumulative rotations on each string.}:, 
\begin{equation}
\sigma(D)^T = 
        \left[\begin{array}{cccccccccccccccccccc}
        a & \pi_{0}(a) & \pi_{0}(a) & \pi_{1}(a) & \pi_{1}(a) & \pi_{3}(a) & \pi_{3}(a) & \pi_{2}(a) & \pi_{3}(a) & \pi_{0}(a) & \pi_{2}(a) & \pi_{3}(a) & \pi_{2}(a) & \pi_{0}(a) & \pi_{2}(a) & \pi_{1}(a) 
        \end{array}\right] \nonumber        
    \end{equation}
    Notice the de Bruijn family is a de Bruijn sequence of order $2$ on the alphabet $\A{C} = \{0,1\}$ and the alternating de Bruijn sequence is a de Bruijn sequence of order $2$ on the alphabet on the alphabet $\A{C} = \{0,1,2,3\}$. This kind of example also appears in~\cite{2x3Cock1988}.
\end{example}

\section{Examples of de Bruijn tori}\label{sec:examples}
For these examples, we write the rotation $\pi_x$ as simply $x$ to simplify notation (i.e. the alphabet $\A{B} = \{0,1,2,3\}$ contains the rotations on strings of length $4$). Furthermore, we display our de Bruijn tori on the binary alphabet $\A{O}=\{0,1\}$ as colored grids using black squares for $1$ and white squares for $0$. For example the de Bruijn torus in Example~\ref{ex:1988deBT} is rendered as:
\begin{center}
    \includegraphics[width=\textwidth]{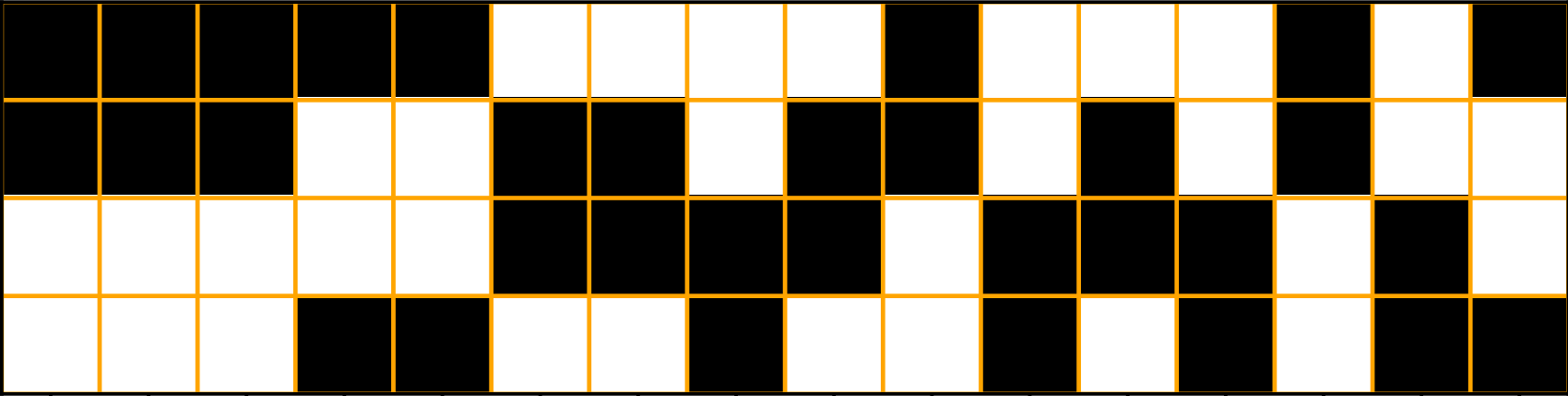}
\end{center}
\begin{example}\label{ex:3x3fixed}
    We now consider construction of a de Bruijn torus with the alphabet $\A{O} = \{0,1\}$ and a $3 \times 3$ window. This gives $\ell = w = 3$ and $|\A{O}| = 2$. From Corollary~\ref{cor:conditionVersion} we have the equations: 
    $rm = 8$ and $rv = 512$. Since de Bruijn families of $r \leq 3$ are impossible (as $l = 3$), we obtain two possibilities. Either $r = 4, m = 2, v = 128$ or $r = 8, m = 1, v = 64$. Note these will generate a $128\times4$ and $64\times8$ size tori respectively.\par
    For $r = 4, m = 2, v = 128$ we must generate a de Bruijn families and alternating de Bruijn sequences. One de Bruijn family\footnote{It is unknown how many de Bruijn families exist.} is: 
    $$\{0001,1110\}$$
    Similarly a  alternating de Bruijn sequences, $\str{D}$, of order $2w-1 = 5$ on the alphabet pair $\A{A} = \{a = 0001,b = 1110\}$ and $\A{B} = \{0,1,2,3\}$\footnote{Recall that $i$ represents $\pi_i$ etc.} is:
   \begin{gather}
    \nonumber 
D = b0b0b0a0a0a0b0a0b1b1b1a1a1a1b1a1b2b2b2a2a2a2b2a2b3b3b3a3a3a3b3a3b0b1b0a1a0a1b0... \\
    \nonumber 
a1b1b0b1a0a1a0b1a0b2b3b2a3a2a3b2a3b3b2b3a2a3a2b3a2b0b2b0a2a0a2b0a2b1b3b1a3a1a3b1a3... \\
    \nonumber
b2b0b2a0a2a0b2a0b3b1b3a1a3a1b3a1b0b3b0a3a0a3b0a3b1b2b1a2a1a2b1a2b2b1b2a1a2a1b2a1b3...\\
    \nonumber
b0b3a0a3a0b3a0
    \label{eq:AdeBS3x3,128x4}
    \end{gather}    
    By Theorem~\ref{th:deBTExist}, $\sigma(\str{D})$ is a de Bruijn torus. Its transpose can be rendered as mentioned above: 
    \begin{center}
    \includegraphics[width=\textwidth]{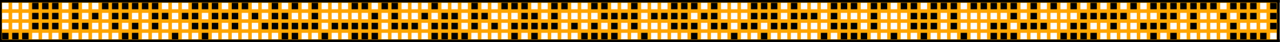}        
    \end{center}\par
    We now consider the other possibility, that is $r = 8, m = 1, v = 64$. Notice the conditions in Remark~\ref{lem:deBTWraps} are satisfied so a  de Bruijn families can be constructed. Since $m=1$, the de Bruijn families can be generated using any de Bruijn sequence of order $\length = 3$. One such de Bruijn sequence is:
    $$00011101$$
    By Lemma~\ref{lem:m=1} the alternating de Bruijn sequence, $\str{D}$, can be constructed from a de Bruijn sequence with the alphabet $\{0,1,2,3,4,5,6,7\}$ of order $2$:
    \begin{gather}
    \nonumber 
D =
a0a0a1a0a2a0a3a0a4a0a5a0a6a0a7a1a1a2a1a3a1a4a1a5a1a6a1a7a2a2a3a2a4a2a5a2a6a2a7... \\
    \nonumber 
a3a3a4a3a5a3a6a3a7a4a4a5a4a6a4a7a5a5a6a5a7a6a6a7a7 
    \label{eq:AdeBS3x3,64x8}
    \end{gather}
    The resulting transpose of de Bruijn torus, $\sigma(\str{D})^T$, is rendered below:
    \begin{center}
    \includegraphics[width=\textwidth]{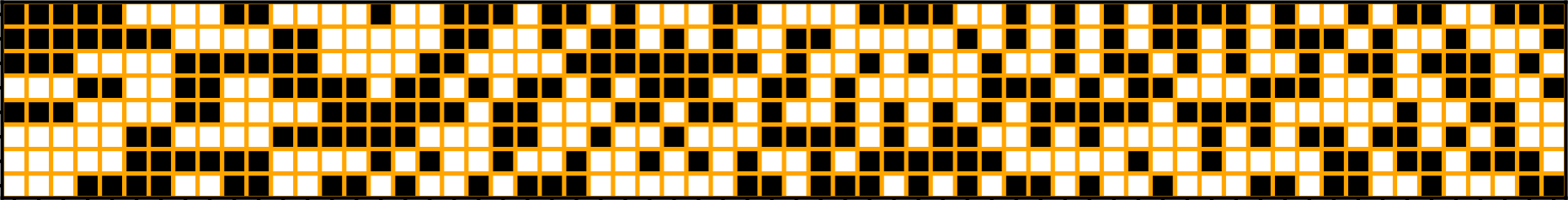}        
    \end{center}\par    
\end{example}
We now generate a de Bruijn torus with a desired size and alphabet.
\begin{example}
    We now consider constructing a de Bruijn torus whose size is $9 \times 9$ and has alphabet $\A{O} = \{0,1,2\}$. This gives $r = v = 9$ and $|\A{O}| = 2$. From Corollary~\ref{cor:conditionVersion}, $9m = 3^{\ell}$ and $81 = 3^{\ell w}$. The only solution to this is $l = w = 2$. Meaning we will construct a de Bruijn torus with window size $2\times2$. Notice the conditions presented in Lemma~\ref{lem:deBTWraps} are satisfied. \par
    We start by generating a de Bruijn sequence of order $\length = 2$ and alphabet $\A{O}$:
    $$\{001021122\}$$
    We now generate an alternating de Bruijn sequence, $\str{D}$, by using Lemma~\ref{lem:m=1} with a a de Bruijn sequence of order $\length = 1$ and alphabet $\A{B} = \{0,1,2,3,4,5,6,7,8\}$:
    $$D = a0a1a2a3a4a5a6a7a8$$
    The following de Bruijn torus, $\sigma(\str{D})$, is shown below:
    \begin{equation}    
        \left[\begin{array}{ccccccccc}
            0& 0& 1& 0& 2& 1& 1& 2& 2 \\
            0& 0& 1& 0& 2& 1& 1& 2& 2 \\ 
            0& 1& 0& 2& 1& 1& 2& 2& 0 \\ 
            0& 2& 1& 1& 2& 2& 0& 0& 1 \\ 
            1& 2& 2& 0& 0& 1& 0& 2& 1 \\ 
            0& 1& 0& 2& 1& 1& 2& 2& 0 \\ 
            1& 2& 2& 0& 0& 1& 0& 2& 1 \\ 
            0& 2& 1& 1& 2& 2& 0& 0& 1 \\ 
            0& 1& 0& 2& 1& 1& 2& 2& 0 \\  
        \end{array}\right]
    \end{equation}    
\end{example}

This exemplifies our construction method's ability to generate multiple tori for desired parameters. The number of valid tori for a parameter set depends both on the number of parameters given, and the number of de Bruijn families and alternating de Bruijn sequences that exist for a set of parameters. 

\section{Concluding remarks}

The work presented here can be extended in many ways. While significant work has been done on the existence of de Bruijn families~\cite{eBugs}, methods to generate de Bruijn families has not been significantly explored. Similarly the number of members in \deBF{\A{A}}{\ringSize}{\ringNum}{\length} would be of interest as it dictates the number of possible de Bruijn tori our construction method produces. Furthermore exploration of `forbidden' de Bruijn tori sizes, that is de Bruijn tori sizes that are impossible for this construction method to construct, could be the focus of future work. Corollary~\ref{cor:conditionVersion} seem to set the groundwork for this.\par 
If we find an efficient method for generating De Bruijn families, we can turn this construction method into an algorithm. At which time exploring the time complexity of such an algorithm would be of interest.

\section{Data Availabilty and Acknowledgements}

The datasets generated during the current study are available from the corresponding author on reasonable request.

We would like to thank the University of Guelph for the opportunity to research these construction methods. We would also like to thank Daniel Ashlock for leading our research into this direction. Matthew Kreitzer was supported by the Ontario Graduate Scholarship (OGS) program. Mihai Nica was supported by the Natural Sciences and Engineering Research Council of Canada (NSERC) Discovery grant RGPIN-2021-02533. Rajesh Pereira was supported by the NSERC Discovery grant RGPIN-2022-04149.

\par
\printbibliography

\end{document}